\documentclass[12pt]{article}
\usepackage[margin=1in]{geometry}  
\usepackage{graphicx}              
\usepackage{amsmath}               
\usepackage{amsthm}                
\usepackage[ansinew]{inputenc}
\usepackage{array,color}
\usepackage{amsxtra}
\usepackage{amstext}
\usepackage{latexsym}
\usepackage{dsfont}
\usepackage[all]{xy}
\usepackage{amsmath,amsfonts,amssymb}
\usepackage{cite}

\usepackage{authblk}

\usepackage{enumitem}  
\usepackage{calc}
\setlist{labelindent=1pt,itemsep=.5em}
\setlist[itemize]{leftmargin=1.2cm}
\setlist[enumerate]{itemindent=0em,leftmargin=1.2cm}
\setlist[enumerate,1]{label={\upshape(\roman*)}}

\allowdisplaybreaks

\makeatletter
\newcommand{\subjclass}[2][2020]{%
  \let\@oldtitle\@title%
  \gdef\@title{\@oldtitle\footnotetext{#1 \emph{Mathematics subject classification}: #2}}%
}
\newcommand{\keywords}[1]{%
  \let\@@oldtitle\@title%
  \gdef\@title{\@@oldtitle\footnotetext{\emph{Keywords}: #1.}}%
}
\makeatother

\newtheorem{thm}{Theorem}[section]
\newtheorem{cor}[thm]{Corollary}
\newtheorem{lem}[thm]{Lemma}
\newtheorem{prop}[thm]{Proposition}
\theoremstyle{definition}
\newtheorem{defn}[thm]{Definition}
\theoremstyle{remark}
\newtheorem{rem}[thm]{Remark}

\numberwithin{equation}{section}

\title{Non-degenerate Killing forms on Hom-Lie superalgebras}
\author{Abdoreza Armakan$^{1}$, Mohammad Reza Farhangdoost$^{1}$, \authorcr
Sergei Silvestrov$^{2}$ \\
\small{$^{1}$Department of Mathematics, College of Sciences, Shiraz University, \authorcr
P.O. Box 71457-44776, Shiraz, Iran. \authorcr
e-mails: r.armakan@shirazu.ac.ir, reza.armakan@gmail.com; farhang@shirazu.ac.ir \authorcr
$^{2}$ Division of Applied Mathematics,
School of Education, Culture and Communication, \authorcr
M\"{a}lardalen University, Box 883, 72123 V\"{a}steras, Sweden. \authorcr
e-mail: sergei.silvestrov@mdh.se}}

\subjclass[2020]{17B61, 17D30, 17B70}
\keywords{Hom-Lie superalgebra, simple Hom-Lie superalgebra, Killing form}

\date{}

\begin{document}

\maketitle


\abstract{In this paper we investigate some important basic properties of simple Hom-Lie superalgebras and show that a Hom-Lie superalgebra does not have any left or right nontrivial ideals. Moreover, we classify invariant bilinear forms on a given simple Hom-Lie superalgebra. Then we study the Killing forms on a Hom-Lie algebra which are examples of the invariant bilinear forms. Making use of the Killing forms, we find conditions for a Hom-Lie superalgebra to be classical. Furthermore, we check the conditions in which the Killing form of a Hom-Lie superalgebra is non-degenerate.}


\section{Introduction} \label{sec:intro}
The investigations of various quantum deformations or $q$-deformations of Lie algebras began a period of rapid expansion in 1980's stimulated by introduction of quantum groups motivated by applications to the quantum Yang-Baxter equation, quantum inverse scattering methods and constructions of the quantum deformations of universal enveloping algebras of semi-simple Lie algebras. Various $q$-deformed Lie algebras have appeared in physical contexts such as string theory, vertex models in conformal field theory, quantum mechanics and quantum field theory in the context of deformations of infinite-dimensional algebras, primarily the Heisenberg algebras, oscillator algebras and Witt and Virasoro algebras. In \cite{AizawaSaito,ChaiElinPop,ChaiIsLukPopPresn,ChaiKuLuk,ChaiPopPres,CurtrZachos1,DamKu,DaskaloyannisGendefVir,Hu,Kassel92,LiuKQuantumCentExt,LiuKQCharQuantWittAlg,LiuKQPhDthesis},
it was in particular discovered that in these $q$-deformations of Witt and Visaroro algebras and some related algebras, some interesting $q$-deformations of Jacobi identities, extending Jacobi identity for Lie algebras, are satisfied. This has been one of the initial motivations for the development of general quasi-deformations and discretizations of Lie algebras of vector fields using more general $\sigma$-derivations (twisted derivations) in \cite{HLS}.

Hom-Lie algebras and more general quasi-Hom-Lie algebras were introduced first by Larsson, Hartwig and Silvestrov \cite{HLS}, where the general quasi-deformations and discretizations of Lie algebras of vector fields using more general $\sigma$-derivations (twisted derivations) and a general method for construction of deformations of Witt and Virasoro type algebras based on twisted derivations have been developed, initially motivated by the $q$-deformed Jacobi identities observed for the $q$-deformed algebras in physics, along with $q$-deformed versions of homological algebra and discrete modifications of differential calculi. Hom-Lie algebras, Hom-Lie superalgebras, Hom-Lie color algebras and more general quasi-Lie algebras and color quasi-Lie algebras where introduced first in \cite{LarssonSilv2005:QuasiLieAlg,LarssonSilv:GradedquasiLiealg,SigSilv:CzechJP2006:GradedquasiLiealgWitt}. Quasi-Lie algebras and color quasi-Lie algebras encompass within the same algebraic framework the quasi-deformations and discretizations of Lie algebras of vector fields by $\sigma$-derivations obeying twisted Leibniz rule, and the well-known generalizations of Lie algebras such as color Lie algebras, the natural generalizations of Lie algebras and Lie superalgebras. In quasi-Lie algebras, the skew-symmetry and the Jacobi identity are twisted by deforming twisting linear maps, with the Jacobi identity in quasi-Lie and quasi-Hom-Lie algebras in general containing six twisted triple bracket terms. In Hom-Lie algebras, the bilinear product satisfies the non-twisted skew-symmetry property as in Lie algebras, and the Hom-Lie algebras Jacobi identity has three terms twisted by a single linear map, reducing to the Lie algebras Jacobi identity when the twisting linear map is the identity map. Hom-Lie admissible algebras have been considered first in \cite{ms:homstructure}, where in particular the Hom-associative algebras have been introduced and shown to be Hom-Lie admissible, that is leading to Hom-Lie algebras using commutator map as new product, and in this sense constituting a natural generalization of associative algebras as Lie admissible algebras. Since the pioneering works \cite{HLS,LarssonSilvJA2005:QuasiHomLieCentExt2cocyid,LarssonSilv:GradedquasiLiealg,LarssonSilv2005:QuasiLieAlg,LarssonSilv:QuasidefSl2,ms:homstructure}, Hom-algebra structures expanded into a popular area with increasing number of publications in various directions. Hom-algebra structures of a given type include their classical counterparts and open broad possibilities for deformations, Hom-algebra extensions of cohomological structures and representations, formal deformations of Hom-associative and Hom-Lie algebras, Hom-Lie admissible Hom-coalgebras, Hom-coalgebras, Hom-Hopf algebras \cite{AmmarEjbehiMakhlouf:homdeformation,BenMakh:Hombiliform,LarssonSilvJA2005:QuasiHomLieCentExt2cocyid,LarssonSilvestrovGLTMPBSpr2009:GenNComplTwistDer,MakhSil:HomHopf,MakhSilv:HomAlgHomCoalg,MakhSilv:HomDeform,Sheng:homrep,Yau:HomolHom,Yau:EnvLieAlg}.
Hom-Lie algebras, Hom-Lie superalgebras and color Hom-Lie algebras have been further investigated in various aspects for example in \cite{AmmarEjbehiMakhlouf:homdeformation,AmmarMakhloufHomLieSupAlg2010,AmmarMakhloufSaadaoui2013:CohlgHomLiesupqdefWittSup,
IJGMMP,ArmakanSilv:envelalgcertaintypescolorHomLie,ArmakanSilvFarh:envelopalgcolhomLiealg,ArmakanSilvFarh:exthomLiecoloralg,
Bakayoko2014:ModulescolorHomPoisson,BakayokoDialo2015:genHomalgebrastr,BakyokoSilvestrov:Homleftsymmetriccolordialgebras,BakyokoSilvestrov:MultiplicnHomLiecoloralg,BakayokoToure2019:genHomalgebrastr,
BenMakh:Hombiliform,CaoChen2012:SplitregularhomLiecoloralg,GuanChenSun:HomLieSuperalgebras,MabroukNcibSilvestrov2020:GenDerRotaBaxterOpsnaryHomNambuSuperalgs,ms:homstructure,MakhSilv:HomDeform,MakhSil:HomHopf,
MakhSilv:HomAlgHomCoalg,Makhlouf2010:ParadigmnonassHomalgHomsuper,ShengBai2014:homLiebialg,SigSilv:CzechJP2006:GradedquasiLiealgWitt,
LarssonSigSilvJGLTA2008:QuasiLiedefFttN,RichardSilvestrovJA2008,RichardSilvestrovGLTbdSpringer2009,ShengChen2013:HomLie2algebras,
Sheng:homrep,SigSilv:GLTbdSpringer2009,SilvestrovParadigmQLieQhomLie2007,Yau2009:HomYangBaxterHomLiequasitring,Yau:EnvLieAlg,
Yau:HomolHom,Yau:HomBial,Yuan2012:HomLiecoloralgstr,ZhouChenMa:GenDerHomLiesuper}.

In this article, we consider properties of simple Hom-Lie superalgebras and show that a Hom-Lie superalgebra does not have any left or right nontrivial ideals, classify invariant bilinear forms on a given simple Hom-Lie superalgebra, study and make use of the Killing forms for finding conditions for a Hom-Lie superalgebra to be classical and check the conditions in which the Killing form of a Hom-Lie superalgebra is non-degenerate. In Section \ref{sec:HomLiesuperalg}, we review some basic notions and definitions for Hom-Lie algebras and Hom-Lie superalgebras. In Section \ref{sec:SimpleHomLiesuperalg}, we present the notion of simple Hom-Lie superalgebras together with some properties of this class of Hom-Lie superalgebras. In Section \ref{sec:Killingforms}, we focus on some properties of a class of invariant bilinear forms, the Killing forms. In particular, we study the Cartan's criterion for Hom-Lie superalgebras, i.e. we find conditions under which the Killing form of a Hom-Lie superalgebra is non-degenerate.

\section{Hom-Lie superalgebras}
\label{sec:HomLiesuperalg}
Recall that a Hom-module is a pair $(M,\alpha)$ consisting of an $\mathbf{k}$-module $M$ and a linear operator $\alpha:M\rightarrow M$. We now recall some definitions to start.
\begin{defn}[\cite{HLS,LarssonSilvJA2005:QuasiHomLieCentExt2cocyid,LarssonSilv:GradedquasiLiealg,LarssonSilv2005:QuasiLieAlg,ms:homstructure}]
A Hom-Lie algebra is a triple $(\mathfrak{g},[\cdot,\cdot], \alpha)$, where $\mathfrak{g} $ is a vector space equipped with a skew-symmetric bilinear map $[\cdot,\cdot]: \mathfrak{g} \times \mathfrak{g} \rightarrow \mathfrak{g}$ and a linear map $\alpha :\mathfrak{g} \rightarrow \mathfrak{g} $ such that
$$[\alpha(x),[y,z]]+[\alpha(y),[z,x]]+[\alpha(z),[x,y]]=0, $$
for all $x,y,z \in \mathfrak{g}$ , which is called Hom-Jacobi identity.
\end{defn}
A Hom-Lie algebra is called a multiplicative Hom-Lie algebra if $\alpha$ is an algebra morphism, i.e. for any $x,y\in \mathfrak{g}$,
$$\alpha([x,y])=[\alpha(x),\alpha(y)].$$

We call a Hom-Lie algebra regular if $\alpha$ is an automorphism.

A sub-vector space $\mathfrak{h}\subset \mathfrak{g}$ is a Hom-Lie sub-algebra of $(\mathfrak{g},[\cdot,\cdot], \alpha)$ if $$\alpha(\mathfrak{h})\subset \mathfrak{h}$$ and $\mathfrak{h}$ is closed under the bracket operation, i.e.
$$[x_{1},x_{2}]_{g}\in \mathfrak{h},$$
for all $x_{1},x_{2}\in \mathfrak{h}.$

Let $(\mathfrak{g},[\cdot,\cdot], \alpha)$ be a multiplicative Hom-Lie algebra. Denote by $\alpha^{k}$ the $k$-times composition of $\alpha$ by itself, for any nonnegative integer $k$,
$$\alpha^{k}=\underbrace{\alpha \circ \dots \circ \alpha}_{k}, $$
where $\alpha^{0}=Id$ and $\alpha^{1}=\alpha$. For a regular Hom-Lie algebra $\mathfrak{g}$, let
$$\alpha^{-k}=\underbrace{\alpha^{-1} \circ \dots \circ \alpha^{-1}}_{k}.$$

\begin{defn}[\cite{BahturinMikhPetrZaicevIDLSbk92,Berezin1987:IntroSuperanal,BerezinKatsGI1970:Liegrcomanticompar,BerezinLeites1975:Supermanifolds,Kac99:LieSuper,SCHbook}]
A Lie superalgebra is a $\mathbb{Z}_{2}$-graded vector space $\mathfrak{g}=\mathfrak{g}_{0}\oplus \mathfrak{g}_{1}$, together with a graded Lie bracket $[\cdot,\cdot]:\mathfrak{g}\times \mathfrak{g} \rightarrow \mathfrak{g}$ of degree zero, i.e. $[\cdot,\cdot]$ is a bilinear map satisfying  $$[\mathfrak{g}_{i},\mathfrak{g}_{j}]\subset \mathfrak{g}_{i+j(mod2)},$$ such that for homogeneous elements $x,y,z \in \mathfrak{g}$, the following identities hold:
\begin{enumerate}[label=\upshape{\arabic*.},left=7pt]
  \item $[x,y]=-(-1)^{|x||y|}[y,x],$
  \item $[x,[y,z]]=[[x,y],z]+(-1)^{|x||y|}[y,[x,z]].$
\end{enumerate}
\end{defn}

In the simliar way as Lie algebras are extended to Lie superalgebras, the Hom-Lie algebras form a subclass of Hom-Lie superalgebras, which is a subclass of color Hom-Lie algebras, which in their turn is a subclass of more general (color) quasi-Lie algebras introduced first in \cite{LarssonSilv2005:QuasiLieAlg,LarssonSilv:GradedquasiLiealg,SigSilv:CzechJP2006:GradedquasiLiealgWitt}.

\begin{defn}[\cite{AmmarMakhloufHomLieSupAlg2010,LarssonSilv2005:QuasiLieAlg,LarssonSilv:GradedquasiLiealg,SigSilv:CzechJP2006:GradedquasiLiealgWitt}]
A Hom-Lie superalgebra is a triple $(\mathfrak{g},[\cdot,\cdot],\alpha)$ consisting of a superspace $\mathfrak{g}$, a bilinear map $[\cdot,\cdot]:\mathfrak{g}\times \mathfrak{g} \rightarrow \mathfrak{g}$ and a superspace homomorphism $\alpha:\mathfrak{g}\rightarrow \mathfrak{g}$, both of degree zero, satisfying
\begin{enumerate}[label=\upshape{\arabic*.},left=7pt]
  \item $[x,y]=-(-1)^{|x||y|}[y,x],$
  \item $(-1)^{|x||z|}[\alpha(x),[y,z]]+(-1)^{|y||x|}[\alpha(y),[z,x]]+(-1)^{|z||y|}[\alpha(z),[x,y]]=0,$
\end{enumerate}
for all homogeneous elements $x,y,z \in \mathfrak{g}$.
\end{defn}
Therefore, a Hom-Lie superalgebra $\mathfrak{g}$ is a $\mathds{Z}_{2}$ graded Hom-algebra $\mathfrak{g}=\mathfrak{g}_{\bar{0}}\oplus \mathfrak{g}_{\bar{1}}$, where we consider $\mathds{Z}_{2}=\{\bar{0}, \bar{1}\}$. One can see that $\mathfrak{g}_{\bar{0}}$ is a Hom-Lie algebra and $\mathfrak{g}_{\bar{1}}$ is a $\mathfrak{g}_{\bar{0}}$-module.

In particular, if for all $x,y \in \mathfrak{g}$ we have
$$\alpha([x,y])=[\alpha(x),\alpha(y)],$$
then we call $(\mathfrak{g},[\cdot,\cdot],\alpha)$, a {\it multiplicative} Hom-Lie superalgebra.

Let $(\mathfrak{g},[\cdot,\cdot],\alpha)$ and $(\mathfrak{g}',[\cdot,\cdot]',\alpha')$ be two Hom-Lie superalgebras. A homomorphism of degree zero $f:\mathfrak{g}\rightarrow \mathfrak{g}'$ is said to be a morphism of Hom-Lie superalgebras if
\begin{enumerate}[label=\upshape{\arabic*.},left=7pt]
  \item $[f(x),f(y)]'=f([x,y])$, for all $x,y \in \mathfrak{g},$
  \item $f \circ \alpha=\alpha' \circ f.$
\end{enumerate}

We can now recall the notion of an $\alpha^{k}$-derivation.

\begin{defn}[\cite{AmmarMakhloufSaadaoui2013:CohlgHomLiesupqdefWittSup}]
Let $(\mathfrak{g},[\cdot,\cdot],\alpha)$ be a Hom-Lie superalgebra. For any nonnegative integer $k$, a linear map $D: \mathfrak{g} \rightarrow \mathfrak{g}$ of degree $d$ is called an $\alpha^{k}$-derivation of the multiplicative Hom-Lie superalgebra $(\mathfrak{g},[\cdot,\cdot], \alpha)$, if
\begin{enumerate}[label=\upshape{\arabic*.},left=7pt]
\item $[D,\alpha]=0$, i.e. $D\circ \alpha= \alpha \circ D,$
\item $D([x,y]_{g})=[D(x),\alpha^{k}(y)]_{\mathfrak{g}}+(-1)^{d|x|}[\alpha^{k}(x),D(y)]_{\mathfrak{g}}$, for all $x,y\in \mathfrak{g}$.
\end{enumerate}
\end{defn}
Denote by $Der_{\alpha^{k}}(\mathfrak{g})$ the set of all $\alpha^{k}$-derivations of the multiplicative Hom-Lie superalgebra $(\mathfrak{g},[\cdot,\cdot], \alpha)$.

For any $x\in \mathfrak{g}$ satisfying $\alpha(x)=x$, define $ad_{k}(x):\mathfrak{g} \rightarrow \mathfrak{g}$ by
$$ad_{k}(x)(y)=[\alpha^{k}(y),x]_{\mathfrak{g}},$$
for all $y\in \mathfrak{g}$.
It is shown in \cite{AmmarMakhloufSaadaoui2013:CohlgHomLiesupqdefWittSup} that $ad_{k}(x)$ is an $\alpha^{k+1}$-derivation, called an inner $\alpha^{k+1}$-derivation.
So,
$$Inn_{\alpha^{k}}(\mathfrak{g})=\{[\alpha^{k-1}(\cdot),x]_{\mathfrak{g}}|x\in \mathfrak{g}, \alpha(x)=x\}.$$
It is also shown that
$$Der(\mathfrak{g})=\bigoplus_{k\geq -1}Der_{\alpha^{k}}(\mathfrak{g})$$
is a Hom-Lie algebra.

\section{Simple Hom-Lie superalgebras}
\label{sec:SimpleHomLiesuperalg}
In this section, simple Hom-Lie superalgebras are introduced and several important properties of this class of Hom-Lie superalgebras are given here. In addition, we manage to give some useful results on a special class of simple Hom-Lie superalgebras, namely, classical simple Hom-Lie superalgebras which will be used in Section \ref{sec:Killingforms}.
\begin{defn}\label{simplehomliesup}
A Hom-Lie superalgebra $\mathfrak{g}$ is called simple if it does not have any nontrivial graded ideals and $[\mathfrak{g},\mathfrak{g}]\neq \{0\}$.
\end{defn}
\begin{rem}
From Definition \ref{simplehomliesup}, one can consider the following properties:
\begin{enumerate}[label=\upshape{(\roman*)},left=7pt]
  \item  A left or right graded Hom-ideal of $\mathfrak{g}$ is automatically a two sided ideal.
  \item  The condition $[\mathfrak{g},\mathfrak{g}]\neq \{0\}$ serves to eliminate the zero-dimensional and the two one-dimensional Hom-Lie superalgebras. It follows that $[\mathfrak{g},\mathfrak{g}]=\mathfrak{g}$.
\end{enumerate}
\end{rem}
According to Definition \ref{simplehomliesup} one might get that a simple Hom-Lie superalgebra might contain nontrivial non-graded ideals, which here is not the case.
\begin{lem}\label{due}
Let $\mathfrak{g}$ be a simple Hom-Lie superalgebra. If $\nu$ is an odd linear mapping of $\mathfrak{g}$ into itself such that
\begin{equation}\label{2.8}
  \nu([a,b])=[a,\nu(b)],
\end{equation}

for all $a,b \in \mathfrak{g},$
then $\nu = 0$.
\end{lem}
\begin{proof}
  One can easily see that the kernel and the image of $\nu$ are graded ideals of $\mathfrak{g}$, hence either $\nu$ is bijective or it equals to zero.
  Suppose that $\nu$ is bijective. Let $a$ and $b$ be any homogeneous elements of $\mathfrak{g}$. If $a$ and $b$ have the same degree, then
  $$[\nu(a),\nu(b)]=-\nu^{2}([a,b]).$$
  But one side of the equation is symmetric in $a,b$ and the other side is skew-symmetric. Therefore, $[a,b]=0$ if $a$ and $b$ are homogenous of the same degree.
  Then again, if $a$ and $b$ are homogenous of not the same degrees, $a$ and $\nu(b)$ are homogenous of the same degree, \eqref{2.8} and our previous result imply that $[a,b]=0$. So, we got that $[\mathfrak{g},\mathfrak{g}]=\{0\}$ which is a contradiction.

\end{proof}
\begin{prop}\label{prop1}
  A simple Hom-Lie superalgebra $\mathfrak{g}$ does not have any left or right nontrivial ideals.
\end{prop}

\begin{proof}
  The linear mapping
  $\kappa:\mathfrak{g}\rightarrow \mathfrak{g}$,
  defined by
  $\kappa(a)=(-1)^{s}a$
  for $a\in \mathfrak{g}_{s}, s\in \mathds{Z}_{2}$, is an automorphism of the Hom-Lie superalgebra $\mathfrak{g}$. For any element $b \in \mathfrak{g}$, its homogeneous component of degree $t\in \mathds{Z}_{2}$ is $\frac{1}{2}(b+(-1)^{t}\kappa(b))$. In particular, a subspace of $\mathfrak{g}$ is $\mathds{Z}_{2}$-graded if and only if it is invariant under $\kappa$.

  Let $I$ be a nontrivial left ideal of $\mathfrak{g}$. Then $\kappa(I)$ is also a left ideal of $\mathfrak{g}$. Therefore, $I+\kappa(I)$ and $I\cap \kappa(I)$ are graded ideals of $\mathfrak{g}$ and hence
  $$I+ \kappa(I)=\mathfrak{g}, \quad I\cap \kappa(I)=\{0\}.$$
  Consequently, the vector space $\mathfrak{g}$ is the direct sum of its subspaces $I$ and $\kappa(I)$. Furthermore, we have
  $$\mathfrak{g}_{t}=\{b+(-1)^{t}\kappa(b)| b\in I\},$$
  for $t\in \mathds{Z}_{2}$. Let $\nu:\mathfrak{g}\rightarrow \mathfrak{g}$ be the linear mapping which is defined by
  $$\nu(b)=b, \quad \nu(\kappa(b))=-\kappa(b),$$
  for all $b\in I$. One can see that $\nu^{2}=id$ and
  $$\nu(\mathfrak{g}_{\bar{0}})=\mathfrak{g}_{\bar{1}}, \quad \nu(\mathfrak{g}_{\bar{1}})=\mathfrak{g}_{\bar{0}}.$$
Moreover, the fact that $I$ and $\kappa(I)$ are left ideals implies the following property for $\nu$:
$$\nu([a,b])=[a,\nu(b)],$$
for all $a,b\in \mathfrak{g}$. But due to Lemma \ref{due}, a mapping $\nu$, having these properties doesn't exist.
One can have the same argument for a right ideal $I$.

\end{proof}
For our next results, we need to recall a useful notation from \cite{SCHbook}. Let $V$ be a finite-dimensional $\mathds{Z}_{2}$-graded vector space and let
$\gamma:V \rightarrow V$
be the linear mapping which satisfies
$\gamma(x)=(-1)^{\xi}x,$
where $x\in V_{\xi}$ and $\xi \in \mathds{Z}_{2}$. A linear form $supertrace$ is defined on the general linear Lie superalgebra $pl(V)$ by
$$str(A)=Tr(\gamma A),$$
for all $A\in pl(V)$. The linear form $str$ is even and $pl(V)$ invariant:
$$str ([A,B])=0,$$
for all $A,B \in pl(V)$ (see \cite{SCHbook}).
\begin{lem}\label{lem2}
Let $\mathfrak{g}$ be a simple Hom-Lie superalgebra. Then
\begin{enumerate}[label=\upshape{(\roman*)},left=7pt]
  \item \label{lem2:i} $[\mathfrak{g}_{\bar{0}},\mathfrak{g}_{\bar{1}}]=\mathfrak{g}_{\bar{1}}$.
  \item \label{lem2:ii} If $\mathfrak{g}_{\bar{1}}\neq \{0\}$, then
  $[\mathfrak{g}_{\bar{1}},\mathfrak{g}_{\bar{1}}]=\mathfrak{g}_{\bar{0}}$ and $\{a\in \mathfrak{g} \mid [a,\mathfrak{g}_{\bar{1}}]=\{0\}\}=\{0\}.$
  Particularly, the adjoint representation of $\mathfrak{g}_{\bar{0}}$ in $\mathfrak{g}_{\bar{1}}$ is faithful.
  \item \label{lem2:iii} If $\rho:a\rightarrow a_{V}$ is a graded representation of $\mathfrak{g}$ in some finite-dimensional graded vector space $V$, then
      $str(a_{v})=0$ for all $a\in \mathfrak{g}.$
\end{enumerate}
\end{lem}
\begin{proof}
  Statements \ref{lem2:i} and \ref{lem2:ii} are proved from the fact that $\mathfrak{g}_{\bar{0}}\oplus[\mathfrak{g}_{\bar{0}},\mathfrak{g}_{\bar{1}}]$, $[\mathfrak{g}_{\bar{1}},\mathfrak{g}_{\bar{1}}]\oplus \mathfrak{g}_{\bar{1}}$ and $\{a\in \mathfrak{g}| [a,\mathfrak{g}_{\bar{1}}]=\{0\}\}$ are graded ideals of $\mathfrak{g}$. Moreover, \ref{lem2:iii} is correct for any Hom-Lie super algebra $\mathfrak{g}$ such that $[\mathfrak{g},\mathfrak{g}]=\mathfrak{g}$.
\end{proof}
\begin{prop} \label{prop:splHomLieinvbf}
  Let $\mathfrak{g}$ be a simple Hom-Lie superalgebra.
  \begin{enumerate}[label=\upshape{(\roman*)},left=7pt]
    \item \label{prop:splHomLieinvbf:i} An invariant bilinear form on $\mathfrak{g}$ is either non-degenerate or equal to zero.
    \item \label{prop:splHomLieinvbf:ii} Every invariant bilinear form on $\mathfrak{g}$ is supersymmetric.
    \item \label{prop:splHomLieinvbf:iii} The invariant bilinear forms on $\mathfrak{g}$ are either all even or else all odd.
    \item \label{prop:splHomLieinvbf:iv} If the field $K$ is algebraically closed, then all invariant bilinear forms on $\mathfrak{g}$ are proportional to each other.
  \end{enumerate}
\end{prop}

\begin{proof}
\ref{prop:splHomLieinvbf:i} Suppose that $\rho$ is an invariant bilinear form on $\mathfrak{g}$. The subspace
$$J=\{y\in \mathfrak{g}| \rho(x,y)=0,\forall x\in \mathfrak{g}\}$$
is a left ideal of $\mathfrak{g}$. Therefore, we get the result by Proposition \ref{prop1}. \\
\ref{prop:splHomLieinvbf:ii} One can easily see it for all Hom-Lie superalgebras like $\mathfrak{g}$ in which $[\mathfrak{g},\mathfrak{g}]=\mathfrak{g}$.\\
\ref{prop:splHomLieinvbf:iii} The homogeneous components of an invariant bilinear form are invariant themselves. So, it suffices to show the following:
      If $\rho$ and $\rho'$ are two homogeneous of not the same degrees invariant bilinear forms on $\mathfrak{g}$ where one of them is non-degenerate, then the other one must be zero. In fact, if $\rho$ is non-degenerate, there exists a unique linear mapping $\kappa$ of $\mathfrak{g}$ into itself, such that for all $x,y \in \mathfrak{g}$,
      $$\rho'(x,y)=\rho(x,\nu(y))$$ It is easy to see that $\nu$ satisfies the conditions in Lemma \ref{due}. Thus, $\nu=0$. \\
\ref{prop:splHomLieinvbf:iv} It is a consequence of \ref{prop:splHomLieinvbf:i}.
\end{proof}
We are now about to study the adjoint representation, which is introduced in \cite{AmmarMakhloufSaadaoui2013:CohlgHomLiesupqdefWittSup}, more precisely. Our main discussion will be the representation of $\mathfrak{g}_{\bar{0}}$ in $\mathfrak{g}_{\bar{0}}$ where we are lead to an interesting class of simple Hom-Lie superalgebras which are known as classical ones. The next lemma is a special case of the propositions which follows it.
\begin{lem}\label{lem5}
Let $\mathfrak{g}$ be s a simple Hom-Lie superalgebra. Suppose that $\mathfrak{g}_{\bar{1}}$ is the direct sum
$$\mathfrak{g}_{\bar{1}}=\bigoplus_{s=1}^{r} \mathfrak{g}_{\frac{s}{1}}$$
of non-zero $\mathfrak{g}_{\bar{0}}$-invariant subspaces $\mathfrak{g}_{\frac{s}{1}}$, $1\leq s\leq r$, $r\geq 1$. Then $r=1$ or $r=2$.
\end{lem}
\begin{proof}
The case $r=1$ is trivial, so let us consider $r=2$. We want to show that
$$J=[\mathfrak{g}_{\frac{1}{1}},\mathfrak{g}_{\frac{1}{1}}]\oplus [\mathfrak{g}_{\frac{1}{1}},[\mathfrak{g}_{\frac{1}{1}},\mathfrak{g}_{\frac{1}{1}}]],$$
is an ideal of $\mathfrak{g}$.
The fact that $J$ is $\mathfrak{g}_{\bar{0}}$-invariant is obvious and we simply see that
$$[\mathfrak{g}_{\frac{1}{1}},J]\subset J.$$
We next remark that
$$[[\mathfrak{g}_{\frac{1}{1}},\mathfrak{g}_{\frac{1}{1}}],\mathfrak{g}_{\frac{2}{1}}] \subset [\mathfrak{g}_{\frac{1}{1}},[\mathfrak{g}_{\frac{1}{1}},\mathfrak{g}_{\frac{2}{1}}]] \subset \mathfrak{g}_{\frac{1}{1}}.$$
Since $\mathfrak{g}_{\frac{2}{1}}$ is $\mathfrak{g}_{\bar{0}}$-invariant, one can get that
$$[\mathfrak{g}_{\frac{2}{1}},[\mathfrak{g}_{\frac{1}{1}},\mathfrak{g}_{\frac{1}{1}}]]=\{0\}.$$
This implies
$$[\mathfrak{g}_{\frac{2}{1}},[\mathfrak{g}_{\frac{1}{1}},[\mathfrak{g}_{\frac{1}{1}},\mathfrak{g}_{\frac{1}{1}}]]] \subset [[\mathfrak{g}_{\frac{2}{1}},\mathfrak{g}_{\frac{1}{1}}],[\mathfrak{g}_{\frac{1}{1}},\mathfrak{g}_{\frac{1}{1}}]] \subset [\mathfrak{g}_{\frac{1}{1}},\mathfrak{g}_{\frac{1}{1}}].$$
One can combine the two equations above to get
$$[\mathfrak{g}_{\frac{2}{1}},J]\subset J.$$
Since $J\neq \mathfrak{g}$ we conclude that $[\mathfrak{g}_{\frac{1}{1}},\mathfrak{g}_{\frac{1}{1}}]=\{0\}$ and similarly $[\mathfrak{g}_{\frac{2}{1}},\mathfrak{g}_{\frac{2}{1}}]=\{0\}$. By part \ref{lem2:ii} of Lemma \ref {lem2} we have
$$[\mathfrak{g}_{\frac{1}{1}},\mathfrak{g}_{\frac{2}{1}}]= [\mathfrak{g}_{\bar{1}},\mathfrak{g}_{\bar{1}}]=\mathfrak{g}_{\bar{0}}.$$
Now suppose that $r\geq 3$. If $s\in\{1,\cdots, r\}$, we get
$$\mathfrak{g}_{\bar{1}}=\mathfrak{g}_{\frac{1}{1}}\oplus \bigoplus_{i\neq s} \mathfrak{g}_{\frac{i}{1}},$$
and the case $r=2$ implies $[\bigoplus_{i\neq s} \mathfrak{g}_{\frac{i}{1}}, \bigoplus_{j\neq s} \mathfrak{g}_{\frac{j}{1}}]=\{0\}$. It follows that $[\mathfrak{g}_{\frac{i}{1}},\mathfrak{g}_{\frac{j}{1}}]\{0\}$, for all $i,j \in \{1,\cdots, r\}$. Since by assumption $\mathfrak{g}_{\bar{1}}\neq \{0\}$, this is a contradiction to the simplicity of $\mathfrak{g}$ (see Lemma \ref{lem2}), because we have $[\mathfrak{g}_{\bar{1}},\mathfrak{g}_{\bar{1}}]=\{0\}$.

\end{proof}
\begin{prop}\label{prop3}
Let $\mathfrak{g}$ be a simple Hom-Lie superalgebra. Suppose that $\mathfrak{g}_{\bar{1}}$ is the sum
$$\mathfrak{g}_{\bar{1}}=\mathfrak{g}_{\frac{1}{1}}+\mathfrak{g}_{\frac{2}{1}}$$
of two proper $\mathfrak{g}_{\bar{0}}$ invariant subspaces $\mathfrak{g}_{\frac{1}{1}}$ and $\mathfrak{g}_{\frac{2}{1}}$.
Therefore, this sum is direct, i.e.
$$\mathfrak{g}_{\frac{1}{1}}\cap \mathfrak{g}_{\frac{2}{1}}=\{0\},$$
and the $\mathfrak{g}_{\bar{0}}$-modules $\mathfrak{g}_{\frac{1}{1}}$ and $\mathfrak{g}_{\frac{2}{1}}$ are irreducible. Furthermore, we get
$$[\mathfrak{g}_{\frac{1}{1}},\mathfrak{g}_{\frac{1}{1}}]= [\mathfrak{g}_{\frac{2}{1}},\mathfrak{g}_{\frac{2}{1}}]=\{0\},\qquad [\mathfrak{g}_{\frac{1}{1}},\mathfrak{g}_{\frac{2}{1}}]=\mathfrak{g}_{\bar{0}}.$$
\end{prop}

\begin{proof}
  We define two sequences $\{V_{n}^{i}\}_{n\geq -1}$, $i=1,2$, of subspaces of $\mathfrak{g}$, as follows. Set
  $$V_{-1}^{i}=\mathfrak{g}_{\bar{1}},\quad V_{0}^{i}= \mathfrak{g}_{\bar{0}}, \quad V_{1}^{i}=\mathfrak{g}_{\frac{i}{1}},$$
  and using induction, we define
  $$V_{n}^{i}=[\mathfrak{g}_{\frac{i}{1}},V_{n-1}^{i}],$$
  where $n\geq2$. It is easy to check that for all $n\geq -1$ we have
 \begin{enumerate}[label=\upshape{\arabic*)},left=7pt]
  \item \label{prprop3:1} $V_{n}^{i}$ is $\mathfrak{g}_{\bar{1}}$-invariant.
  \item \label{prprop3:2} $[\mathfrak{g}_{\bar{1}},V_{n+1}^{i}]\subset V_{n}^{i} $,
  \item \label{prprop3:3} $V_{n+2}^{i} \subset V_{n}^{i}$.
 \end{enumerate}
From \ref{prprop3:3} one sees that there exists an integer $m\geq1$ such that
$$V_{2m+2}^{i}=V_{2m}^{i}.$$
Then, using \ref{prprop3:1} and \ref{prprop3:2} we deduce that $V_{2m}^{i}\oplus V_{2m+1}^{i}$ is a graded ideal of $\mathfrak{g}$ which must be equal to $\{0\}$. Thus one can conclude that $V_{n}^{i}=\{0\}$, for a sufficiently large positive integer $n$. Now let $r\geq 0$ be any positive integer and define
\begin{eqnarray}\label{480a}
J^{r}_{0} &=& \sum^{r}_{s=0}(V_{2(r-s)}^{1}\cap V_{2s}^{2}) \\
\label{480b}
J^{r}_{1} &=& \sum^{r+1}_{s=0}(V_{2(r-s)+1}^{1}\cap V_{2s-1}^{2}) \\
\label{480c}
J^{r} &=& J^{r}_{0}\oplus J^{r}_{1}.
\end{eqnarray}
Using \ref{prprop3:1}, \ref{prprop3:2} and \ref{prprop3:3} it is easy to see that $J^{r}$ is a graded ideal of $\mathfrak{g}$, for every $r\geq 0$. We remark that
$$J^{0}_{0}=\mathfrak{g}_{\bar{0}}, \qquad J^{0}_{1}=\mathfrak{g}_{\frac{1}{1}}+\mathfrak{g}_{\frac{2}{1}}=\mathfrak{g}_{\bar{1}}.$$
Evidently,
$$J^{r}_{1}\subset (V_{2r+1}^{1}+\mathfrak{g}_{\frac{2}{1}})\cap (\mathfrak{g}_{\frac{1}{1}}+V_{2r+1}^{2}).$$
Therefore, if $V_{2r+1}^{1}=\{0\}$ or $V_{2r+1}^{2}=\{0\}$, then $J^{r}_{1}\neq \mathfrak{g}_{\bar{1}}$ and hence $J^{r}=\{0\}$.
Now let $R$ be the smallest integer such that
\begin{equation}\label{2.51}
V_{2(r-s)+1}^{1}\cap V_{2s-1}^{2}=\{0\},
\end{equation}
where $1\leq s\leq r$. Then we have $V_{2R-1}^{1}$ and $V_{2R-1}^{2}$ are no equal to $\{0\}$, since otherwise $R\geq 2$ and $J^{R-1}=\{0\}$, hence $R$ would not be minimal. Particularly, one can conclude that $J^{R-1}=\mathfrak{g}$ and as a consequence, $J^{R-1}_{1}=\mathfrak{g}_{\bar{1}}$. On the other hand it follows from \eqref{2.51} that the sum defining $J^{R-1}_{1}$ is direct. Since we already know that the two terms $V_{2R-1}^{1}$ and $V_{2R-1}^{2}$ of this sum are not equal to $\{0\}$, we deduce from Lemma \ref{lem5} that all the remaining terms must be $\{0\}$. This implies that $R=1$, otherwise it would not be minimal. Thus we have shown that $\mathfrak{g}_{\bar{1}}$ is the direct sum of $V_{1}^{1}=\mathfrak{g}_{\frac{1}{1}}$ and $V_{1}^{2}=\mathfrak{g}_{\frac{2}{1}}$. It is now easy to check that these two are irreducible $\mathfrak{g}_{\bar{0}}$-modules.
\end{proof}

\begin{rem}\label{rem4}
Let $\mathfrak{g}$ be a simple Hom-Lie superalgebra. Then there exist only the following possibilities:
\begin{enumerate}[label=\upshape{(\roman*)},left=7pt]
  \item  The $\mathfrak{g}_{\bar{0}}$-module $\mathfrak{g}_{\bar{1}}$ is completely reducible. Then $\mathfrak{g}_{\bar{1}}$ is decomposed into at most two irreducible components.
  \item The $\mathfrak{g}_{\bar{0}}$-module $\mathfrak{g}_{\bar{1}}$ is not completely reducible. In this case there exists a unique proper $\mathfrak{g}_{\bar{0}}$-submodule of $\mathfrak{g}_{\bar{1}}$ which contains all proper $\mathfrak{g}_{\bar{0}}$-submodules of $\mathfrak{g}_{\bar{1}}$.
\end{enumerate}
\end{rem}
We shall see that these two possibilities do occur.
\begin{defn}
A simple Hom-Lie superalgebra $\mathfrak{g}$ is called classical if the $\mathfrak{g}_{\bar{0}}$-module $\mathfrak{g}_{\bar{1}}$ is completely reducible.
\end{defn}

\begin{thm}\label{corth1}
Suppose that the field $K$ is algebraically closed. Let $\mathfrak{g}$ be a classical simple Hom-Lie superalgebra such that the center $\mathfrak{g}_{\frac{a}{0}}$ of $\mathfrak{g}_{\bar{0}}$ is nontrivial. Then $\dim  \mathfrak{g}_{\frac{a}{0}}=1$ and the $\mathfrak{g}_{\bar{0}}$-module $\mathfrak{g}_{\bar{1}}$ decomposes into the direct sum of two irreducible $\mathfrak{g}_{\bar{0}}$-modules,
$$\mathfrak{g}_{\bar{1}}=\mathfrak{g}_{\frac{1}{1}}\oplus \mathfrak{g}_{\frac{2}{1}}.$$
Moreover, there exists a unique element $c\in \mathfrak{g}_{\frac{a}{0}}$ such that
$$[c,x]=(-1)^{r}x,$$
for all $x\in \mathfrak{g}_{\frac{r}{1}}$, $r=1,2$.
\end{thm}
\begin{proof}
Noting Remark \ref{rem4}, the first part of the proof is trivial. For the second part, suppose that $\mathfrak{g}_{\bar{1}}$ is irreducible. If $a\in \mathfrak{g}_{\frac{a}{0}}$, then there exists an element $\lambda \in K$ such that
$$[a, x]= \lambda x,$$
for all $x \in \mathfrak{g}_{\bar{1}}$, which implies
$$2\alpha \mathfrak{g}_{\bar{0}}= 2\lambda[\mathfrak{g}_{\bar{1}},\mathfrak{g}_{\bar{1}}]=[A, [\mathfrak{g}_{\bar{1}},\mathfrak{g}_{\bar{1}}]]=\{0\}.$$
by Lemma \ref{lem2}, so $\lambda=0$. Again by Lemma \ref{lem2}, the representation of $\mathfrak{g}_{\bar{0}}$ in $\mathfrak{g}_{\bar{1}}$ is faithful, therefore $\mathfrak{g}_{\frac{a}{0}}=\{0\}$, as a contradiction to the assumption. Now, we have that $\mathfrak{g}_{\bar{0}}$-module $\mathfrak{g}_{\bar{1}}$ is completely reducible. Hence from Proposition \ref{prop3} we have
$$\mathfrak{g}_{\bar{1}}=\mathfrak{g}_{\frac{1}{1}}\oplus \mathfrak{g}_{\frac{2}{1}}.$$
If $a$ is an arbitrary element of $\mathfrak{g}_{\frac{a}{0}}$, then there exist two elements $\lambda_{r}, r=1,2$ in $K$ such that
$$[a,x_{r}]=\lambda_{r}x_{r},$$
for all $x_{r}\in \mathfrak{g}_{\frac{r}{1}}, r=1,2$. Again from Proposition \ref{prop3} it follows that
$$(\lambda_{1}+\lambda_{2})\mathfrak{g}_{\bar{0}}= (\lambda_{1}+\lambda_{2}) [\mathfrak{g}_{\frac{1}{1}},\mathfrak{g}_{\frac{2}{1}}] = [a,[\mathfrak{g}_{\frac{1}{1}},\mathfrak{g}_{\frac{2}{1}}]]=\{0\}. $$
Therefore, $(\lambda_{1}+\lambda_{2})=0$. Meanwhile, the representation of $\mathfrak{g}_{\bar{0}}$ in $\mathfrak{g}_{\bar{1}}$ is faithful by Lemma \ref{lem2} which implies that $\dim  \mathfrak{g}_{\frac{a}{0}}\leq 1.$
\end{proof}
\section{The Killing forms}
\label{sec:Killingforms}
Let $\mathfrak{g}$ be a Hom-Lie superalgebra and let $\phi$ be a bilinear form on $\mathfrak{g}$. Recall that $\phi$ is called invariant if
$$\phi([a,b],c)=\phi(a,[b,c]),$$
for all $a,b,c \in \mathfrak{g}$. Important examples of invariant bilinear forms are the Killing form and more generally, the bilinear forms associated with the finite-dimensional graded $\mathfrak{g}$-modules\cite{SCHbook}. These bilinear forms are even and being supersymmetric is quite a "normal feature" of invariant bilinear forms on $\mathfrak{g}$, since we can prove the following proposition.

\begin{prop}\label{prop3.1}
Let $\mathfrak{g}$ be a Hom-Lie superalgebra such that $[\mathfrak{g},\mathfrak{g}]=\mathfrak{g}$. Then every invariant bilinear form on $\mathfrak{g}$ is supersymmetric.
\end{prop}

\begin{proof}
Let $\phi$ be an invariant bilinear form on $\mathfrak{g}$ and let $a\in \mathfrak{g}_{x}$, $b\in \mathfrak{g}_{y}$ and $c\in \mathfrak{g}_{z}$, where $x,y,z \in \mathds{Z}_{2}$. Then it is easy to check that
\begin{equation}\label{3.2}
  \phi(a,[b,c])=(-1)^{x(y+z)}\phi([b,c],a).
\end{equation}
Since $[\mathfrak{g},\mathfrak{g}]=\mathfrak{g}$, our proposition is proved.

\end{proof}

In connection with the bilinear forms associated with graded $\mathfrak{g}$-modules, the following proposition is quite interesting.

\begin{prop}\label{prop2}
Let $\mathfrak{g}$ be a Hom-Lie superalgebra and let $\phi$ be an invariant bilinear form on $\mathfrak{g}$ which is associated with some finite-dimensional graded $\mathfrak{g}$-module. If $\phi$ is non-degenerate, then the Hom-Lie algebra $\mathfrak{g}_{\bar{0}}$ is reductive.
\end{prop}

\begin{proof}
  Suppose that $\phi$ is associated with the graded $\mathfrak{g}$-module $V$. Let $\phi^{x}$, $x\in \mathds{Z}_{2}$ be the bilinear form on $\mathfrak{g}_{\bar{0}}$ which is associated with the $\mathfrak{g}_{\bar{0}}$-module $V_{x}$.
  Then we have
\begin{equation}\label{3.3}
\phi(p,q)=\phi^{\bar{0}}(p,q)-\phi^{\bar{1}}(p,q),
\end{equation}
for all $p,q \in \mathfrak{g}_{\bar{0}}$. Set
\begin{equation}\label{3.4}
J^{x}=\{q\in \mathfrak{g}_{\bar{0}}| \phi^{x}(q,\mathfrak{g}_{\bar{0}})=\{0\}\},
\end{equation}
for any $\alpha\in \mathds{Z}_{2}$. Since $\phi$ is even and non degenerate, \eqref{3.3} implies $J^{\bar{0}}\cap J^{\bar{1}}=\{0\}.$ Then a standard result from Hom-Lie algebra theory says that $\mathfrak{g}_{^{\bar{0}}}$ has to be reductive.
\end{proof}
\begin{cor}\label{badprop2}
  Let $\mathfrak{g}$ be a simple Hom-Lie superalgebra whose Killing form is non-degenerate. Then $\mathfrak{g}$ is classical.
\end{cor}

The following two propositions contain some information on the existence of non-degenerate invariant bilinear forms on a Hom-Lie superalgebra.

\begin{prop}
  Let $\mathfrak{g}$ be a classical simple Hom-Lie superalgebra such that the center of $\mathfrak{g}_{\bar{0}}$ is nontrivial. Then the Killing form of $\mathfrak{g}$ is non-degenerate.
\end{prop}

\begin{proof}
Denote the Killing form of $\mathfrak{g}$ by $\phi$. If the field $K$ is algebraically closed, the proposition follows directly from
Corollary \ref{corth1}.

Suppose now that the field $K$ is arbitrary. Let $E$ be an algebraically closed extension field of $K$ and let $\hat{\mathfrak{g}}=\otimes \mathfrak{g}$ denote the Hom-Lie superalgebra which is obtained from $\mathfrak{g}$ be extension of base field from $K$ to $E$. We know that the Hom-Lie superalgebra $\hat{\mathfrak{g}}$ is the direct sum of graded ideals $\hat{\mathfrak{g}}^{r}$, $1\leq r\leq t$, which are all classical simple Hom-Lie superalgebras. Since at least one of the Hom-Lie algebras $\hat{\mathfrak{g}}\frac{r}{0}$ has a nontrivial center, the first part of our proof, combined with the subsequent lemma, shows that the Killing form of $\hat{\mathfrak{g}}$ is non-zero. But then the Killing form $\phi$ of $\mathfrak{g}$ is also non-zero, hence $\phi$ is non-degenerate.
\end{proof}

\begin{lem}\label{lem1}
  Let $\mathfrak{g}$ be a Hom-Lie superalgebra and let $\mathfrak{g}^{'}$ be a graded ideal of $\mathfrak{g}$. If $\phi$ {\rm (}resp. $\phi^{'}${\rm )} is the Killing form of $\mathfrak{g}$ {\rm (}resp. $\mathfrak{g}'${\rm )}. Then the restriction of $\phi$ to $\mathfrak{g}'$ is equal to $\phi^{'}$. If $\mathfrak{g}'$ and $\mathfrak{g}''$ are two graded ideals of $\mathfrak{g}$ such that $[\mathfrak{g}',\mathfrak{g}'']=\{0\}$, then these ideals are orthogonal with respect to $\phi$.
\end{lem}

\begin{proof}
  The proof is obvious.
\end{proof}
\begin{prop}
Let $\mathfrak{g}$ be a Hom-Lie superalgebra whose Killing form is non-degenerate. Then every $\alpha^{k}$ derivation of $\mathfrak{g}$ is an inner derivation.
\end{prop}

\begin{proof}
  Let $\tilde{\phi}$ be the Killing form of $Der_{\alpha^{k}}(\mathfrak{g})$. From \cite{AmmarMakhloufSaadaoui2013:CohlgHomLiesupqdefWittSup} we know that
  $$ad_{k}:\mathfrak{g} \rightarrow Der_{\alpha^{k}}(\mathfrak{g}),$$
  is a homomorphism of Hom-Lie superalgebras. Since the Killing form of $\mathfrak{g}$ is non-degenerate, this homomorphism is injective. Furthermore, the image $ad_{k}(\mathfrak{g})$ of $\mathfrak{g}$ is a Hom-graded ideal of $Der_{\alpha^{k}}(\mathfrak{g})$. According to Lemma \ref{lem1}, the restriction of $\tilde{\phi}$ to $ad_{k}(\mathfrak{g})$ is equal to the Killing form of $ad_{k}(\mathfrak{g})$. Thus, this restriction is non-degenerate.

  On the other hand, we have
  \begin{equation}\label{3.16}
  \tilde{\phi}(ad_{k}(d(a)),ad_{k}(b))= \tilde{\phi}(d, ad_{k}[a,b]),
  \end{equation}
  for all $d \in Der_{\alpha^{k}}(\mathfrak{g})$ and $a,b \in \mathfrak{g}$. Now, let
  \begin{equation}\label{3.17}
    J=\{d\in Der_{\alpha^{k}}(\mathfrak{g})| \tilde{\phi}(d,ad_{k}{\mathfrak{g}})=\{0\}\}.
  \end{equation}
  Then \eqref{3.16} and the foregoing remark imply that $d(a)=0$, for all $d\in J$ and $a\in \mathfrak{g}$. Therefore, we have shown that $J=\{0\}$ and the proof is complete.
\end{proof}

\begin{thm}\label{th1}
  Let $\mathfrak{g}$ be a Hom-Lie superalgebra containing no non-zero commutative graded ideals. Suppose that there exists a homogeneous non-degenerate invariant bilinear form $\phi$ on $\mathfrak{g}$.
  In this case $\mathfrak{g}$ has only a finite number of minimal graded ideals $\mathfrak{g}^{r}$, $1\leq r\leq t$, and $\mathfrak{g}$ is their direct sum. The ideals $\mathfrak{g}^{r}$ are simple Hom-Lie superalgebras and they are mutually orthogonal with respect to the bilinear form $\phi$ which is supersymmetric. Any left or right ideal of $\mathfrak{g}$ is graded and equal to $\bigoplus_{r\in R}\mathfrak{g}^{r}$ with a suitable set $R$, where $R\subset \{1,\ldots,t\}$.
\end{thm}
\begin{proof}
If $J$ is a minimal graded ideal of $\mathfrak{g}$, then
\begin{equation}\label{3.19}
  J^{\perp}=\{a\in \mathfrak{g}| \phi(a,J)=\{0\}\}
\end{equation}
is also a graded ideal of $\mathfrak{g}$. Now suppose that $J$ is a non-zero minimal graded ideal of $\mathfrak{g}$. Then $J^{\perp}\cap J$ is a graded ideal of $\mathfrak{g}$ and therefore it is equal to $\{0\}$ or to $J$. If $J^{\perp}\cap J=J$, we have $J \subset J^{\perp}$. So
\begin{equation}\label{3.20}
\phi(\mathfrak{g},[J,J]))=\phi([\mathfrak{g},J],J)\subset \phi(J^{\perp},J)=\{0\}.
\end{equation}
Since $\phi$ is non-degenerate, we conclude that $[J,J]=\{0\}$. But, according to the assumptions, there is no non-zero commutative graded ideals contained in $\mathfrak{g}$. This is a contradiction, thus
\begin{eqnarray}\label{3.21}
  &\mathfrak{g}= J^{\perp}\oplus J  \ , &\\
  \label{3.22}
 & [J^{\perp},J]\subset J^{\perp}\cap J=\{0\}.&
\end{eqnarray}
Therefore, each graded ideal of $J$ or $J^{\perp}$, is a graded ideal of $\mathfrak{g}$. This shows that $J$ is a simple Hom-Lie superalgebra and that $J^{\perp}$ does not contain any non-zero commutative graded ideals. Moreover, it is obvious that $\phi^{\perp}$, which is the restriction of $\phi$ to $J^{\perp}$ is non-degenerate. As a consequence, the pair $(J^{\perp},\phi^{\perp})$ satisfies the same conditions as $(\mathfrak{g},\phi)$ does. Using induction on $\dim \mathfrak{g}$, one can consider
$$\mathfrak{g}=\bigoplus^{t}_{r=1}\mathfrak{g}^{r},$$
where the $\mathfrak{g}^{r}$, are minimal (and hence simple) graded ideals of $\mathfrak{g}$ which are mutually orthogonal with respect to $\phi$ in the sense that
\begin{equation}\label{3.24}
  \phi(\mathfrak{g}^{r},\mathfrak{g}^{s})=\{0\},
\end{equation}
where $1\leq r< s\leq t$. Consequently, we get $[\mathfrak{g},\mathfrak{g}]=\mathfrak{g}$. Therefore, The bilinear form $\phi$ is supersymmetric due to Proposition \ref{prop3.1}. Thus, \eqref{3.24} may be generalized to
$$\phi(\mathfrak{g}^{r},\mathfrak{g}^{s})=\{0\},$$
where $r,s \in \{1,\ldots,t\};r\neq s$.

Now let $\mathfrak{g}'$ be a (non not necessarily graded) left ideal of $\mathfrak{g}$. If $s\in \{1,\ldots,t\}$, the intersection $\mathfrak{g}^{s}\cap \mathfrak{g}'$ is a left ideal of $\mathfrak{g}^{s}$, so $\mathfrak{g}^{s}\cap \mathfrak{g}'$ is equal to $\{0\}$ or to $\mathfrak{g}^{s}$ by Proposition \ref{prop1}. In the first case, we have
$$[\mathfrak{g}^{s}, \mathfrak{g}']=\mathfrak{g}^{s}\cap \mathfrak{g}'=\{0\}$$
and hence
$$\mathfrak{g}'\subset \bigoplus_{r\neq s}\mathfrak{g}^{r}.$$
In the second case, $\mathfrak{g}^{s}\subset \mathfrak{g}'$. This, immediately implies that
$$\mathfrak{g}'=\bigoplus_{r\in R}\mathfrak{g}^{r}$$
where $R$ is a suitable subset of $\{1,\ldots,t\}$. In particular, each minimal graded ideal of $\mathfrak{g}$ is equal to some $\mathfrak{g}^{r}$.
The case in which $\mathfrak{g}'$ is a right ideal is similar.
\end{proof}

\begin{cor}\label{badth1}
  The Killing form of a Hom-Lie superalgebra is non-degenerate if and only if $\mathfrak{g}$ is the direct product of classical simple Hom-Lie superalgebras whose Killing forms are non-degenerate. This is the case when the Hom-Lie algebra $\mathfrak{g}_{\bar{0}}$ is reductive and the representation of $\mathfrak{g}_{\bar{0}}$ in $\mathfrak{g}_{\bar{1}}$ is completely reducible.
\end{cor}

\begin{proof}
Let $\mathfrak{g}$ be a Hom-Lie superalgebra whose Killing form $\varphi$ is non-degenerate. We need to use Theorem \ref{th1}, so we must show that $\mathfrak{g}$ doesn't contain any non-zero commutative graded ideals. Let $J$ be a commutative graded ideal of $\mathfrak{g}$. Then we get
\begin{equation}\label{3.29}
[\mathfrak{g},[J,\mathfrak{g}]]\subset J, \quad [\mathfrak{g},[J,J]]=\{0\}.
\end{equation}
It is easy to see that this implies
$\varphi(\mathfrak{g},J)=\{0\}.$
Therefore $J=\{0\}$, since $\varphi$ is non-degenerate which allows us to use Theorem \ref{th1}. The simple graded ideals $\mathfrak{g}^{r}$ are orthogonal with respect to $\varphi$, so the restriction of $\varphi$ to $\mathfrak{g}^{r}$ is non-degenerate, which is the Killing form of $\mathfrak{g}^{r}$, by Lemma \ref{lem1}. In particular it follows that the simple Hom-Lie superalgebras $\mathfrak{g}^{r}$ are classical (see Corollary \ref{badprop2}). This shows that the Hom-Lie algebra $\mathfrak{g}_{\bar{0}}$ is reductive (by Proposition \ref{prop2}) and that the representation of $\mathfrak{g}_{\bar{0}}$ in $\mathfrak{g}_{\bar{1}}$ is completely reducible.
The converse is clear by Lemma \ref{lem1}.

\end{proof}

\section{Acknowledgement}
We would like to thank Iran National Science Foundation (INSF) for support of the research in this paper.


\end{document}